\documentclass[11pt,reqno,twoside]{amsart}
\usepackage[utf8]{inputenc}
\usepackage[T1]{fontenc}
\usepackage{amsmath}
\usepackage[colorlinks=true, citecolor=black, urlcolor=black, linkcolor=black, pdfborder={0,0,0}]{hyperref}
\usepackage{amsmath} % assumes amsmath package installed
\usepackage{amssymb}  % assumes amsmath package installed
\usepackage{color}
\usepackage[capitalise]{cleveref}
\usepackage[textsize=small]{todonotes}

\usepackage{amsthm}
\usepackage[dvips,bottom=1.4in,right=1in,top=1in, left=1in]{geometry}

\newtheorem{theorem}{Theorem}

\newtheorem{lemma}{Lemma}
\newtheorem{corollary}{Corollary}

\crefformat{equation}{(#2#1#3)}

\newcommand{\pcal}{\mathcal{P}}
\newcommand{\rmD}{\mathrm{D}}
\newcommand{\gradapprox}{{\mathrm{D}_{\mu}}}
\title{%\LARGE \bf
Differentiability and Control of a Model for Granular Material Accumulation
}
\author{Rafael Arndt and Carlos N. Rautenberg}
\thanks{C. N. R. was supported by NSF grant DMS-2012391.}

\thanks{R. Arndt ({\tt\small tarndt@gmu.edu}) and C. N. Rautenberg ({\tt\small crautenb@gmu.edu}) are members of the Department of Mathematical Sciences at George Mason University, and of the Center for Mathematics and Artificial Intelligence (CMAI), Fairfax VA, USA.  }%

\begin{document}

%\maketitle
%\thispagestyle{empty}
%\pagestyle{empty}

\begin{abstract}
We consider differentiability issues associated to the problem of minimizing the accumulation of a granular cohesionless material on a certain surface. The design variable or control is determined by source locations and intensity thereof. The control problem is described by an optimization problem in function space and constrained by a  variational inequality or a non-smooth equation.  We address a regularization approach via a family of nonlinear partial differential equations, and provide a novel result of Newton differentiability of the control-to-state map. Further, we discuss solution algorithms for the state equation as well as for the optimization problem.
\end{abstract}

\maketitle
%	\tableofcontents

\section{Introduction}
The growth description of piles made of granular cohensionless and homogeneous materials is a complex, nonconvex and nonsmooth process. In this problem, we assume that material is deposited by a known source on a supporting structure $u_0$ that may not be flat. The material is characterized by its angle of repose $\alpha>0$: The steepest angle at which a sloping surface formed from a point source of material is stable. The source (assumed constant in time) is given by $g:\Omega \to \mathbb{R}$, where $\Omega\subset \mathbb{R}^\mathrm{d}$ with $\mathrm{d}=1,2$, and represents the {(density)}  rate of a granular material being deposited on the smooth supporting structure $u_0:\Omega\to \mathbb{R}$  with $u_0|_{\partial \Omega}=0$. The latter boundary condition translates into the ability of material to escape the surface freely if it reaches the boundary $\partial \Omega$. In  the limit $\mathrm{time}\to\infty$, the function $u:\Omega\to\mathbb{R}$ describing the surface of the outmost layer of material is approximated by the solution to the stationary variational inequality: Find  $u\in K$ such that 
\begin{align}\label{eq:original}
	\langle -\epsilon\Delta u - f, v - u\rangle_{H^{-1},H_0^1} \geq 0,
\end{align}
for all $v\in K$ with 
\begin{equation*}
	 K:=\{v\in H_0^1(\Omega) : |\mathrm{D} v| \leq \varphi \: \text{ a.e.} \}
\end{equation*}
and where $f=g+\epsilon \Delta u_0$, and $ H_0^1(\Omega)$ is the space of $L^2(\Omega)$ functions such that their weak gradients belong to $L^2(\Omega)^{\mathrm{d}}$, together with function values vanishing at the boundary $\partial\Omega$ in the sense of the trace; see \cite{adams_fournier_2003}.

We assume that $\mathrm{D}$ is either the weak gradient or approximation thereof, and $0<\epsilon\ll 1$. The function $\varphi:\Omega \to \mathbb{R}$ is strictly and uniformly above zero, i.e., 
\begin{equation}\label{eq:nu}
	\varphi(x)\geq \nu>0,
\end{equation}
for almost all $x\in \Omega$. If the pile is homogeneous and $|\nabla u_0|\leq \tan(\alpha)$, then  $\varphi\equiv \tan(\alpha)$. In the case of an inhomogeneous pile (more than one material present),  $\alpha$ is not longer a constant and neither is $\varphi$. There is a further more complex case (not treated within this paper) when  $|\nabla u_0|> \tan(\alpha)$ on a positive measure set within $\Omega$; in this case $\varphi$ is actually dependent on $u$ and the problem is a \emph{quasi-variational inequality}. This approach was pioneered by Prigozhin \cite{Prigozhin1994,Prigozhin1996,Prigozhin1996a}.

A control problem of interest associated to problem \eqref{eq:original} corresponds to the selection of a source of material $f$ so that the accumulation of material $u=u(f)$ is close to some desired structure $u^d$ while $f$ also is minimized in some sense. This leads to the minimization of the functional
\begin{equation}\label{eq:J}
		J(u(f), f) := \frac12\int_\Omega|u(f)-u^d|^2\,\textup{d}x + \lambda \|f\|^2_{Y'},
\end{equation} 
for some space $Y'\subset H^{-1}(\Omega)$ where $H^{-1}(\Omega)$ is the topological dual to $H_0^{1}(\Omega)$. In the same vein, the design of algorithms for the minimization of \eqref{eq:J} and the study of optimality conditions require information on  the differentiability of the map $f\mapsto u(f)$. This is an extremely complex task due to the constraint $K$. In light of this, problem \eqref{eq:original} is replaced by the regularized version 
\begin{align}\label{eq:reg}
	 -\epsilon\Delta u +\gamma\mathcal{P}(u) = f \quad \text{in } H^{-1}(\Omega),
\end{align}  
for some $\mathcal{P}:H_0^1(\Omega)\to H^{-1}(\Omega)$ monotone, nonsmooth, and vanishing at $K$, whose specific form is given later in the paper. Further, $\gamma>0$ and we recover (in some sense explained later) the original problem when $\gamma\to \infty$.

The abstract version of the problem above can be formulated as the following minimization problem with a non-smooth equation as constraint:
\begin{equation}\label{eq:control_problem}\tag{P}
\begin{aligned}
&\mathrm{Minimize }\:\:\:  J(u, f)\\
&\text{subject to } A(u) + F(u) = f
\end{aligned}
\end{equation}
for some objective function \(J:X\times Y'\to\mathbb{R}\), with $Y'\subset X'$, and
where \(A:X\to X'\) is a strongly monotone, and continuously Fr\'{e}chet differentiable operator. Additionally, \(F:X\to X'\) is monotone, and Newton differentiable (see the next section for a definition). Existence of solutions to \eqref{eq:control_problem} is available under mild conditions.

In this paper we study the Newton differentiability properties of the solution map $f\mapsto u(f)$ associated to the constraint in \eqref{eq:control_problem} and how this permeates to $f\mapsto J(u(f), f)$. In particular, this implies that \eqref{eq:control_problem} is suitable to be tackled by  semismooth Newton approaches.

The rest of the paper is organized as follows. In section \ref{sec:abstract_results} we study results involving differentiability and monotonicity, and further establish an abstract result, Theorem \ref{thm:thm1}, for the characterization of the sensitivity of the control-to-state map.  Subsequently, in section \ref{sec:appl}, we apply the abstract results to our specific application. There we improve known results for the Newton differentiability of convex regularization of gradient type constraints, and also consider analogous results for operators approximating the gradient. Several lemmata are proven that culminate in Theorem \ref{thm:ftou} which establishes the control-to-state differentiability result for the application example. The paper ends with a discussion on solution algorithms and future research directions.

\section{Preliminaries and theoretical results}\label{sec:abstract_results}

We assume throughout this section that $X$ is a reflexive and real Banach space, and additionally assume that the control space $Y'$ is such that $Y'\subset X'$ with $Y$ a reflexive and real Banach space. The typical example of application here is $X=H_0^1(\Omega)$ and $Y'\simeq Y = L^2(\Omega)$.

We start with monotonicity definitions used throughout the paper. We say that $A:X\to X'$ is \emph{strongly monotone} if there exist $c>0$ and $q>1$ such that
\begin{equation}\label{eq:monot}
\left\langle A (u+h) - A(u), h\right\rangle \geq c \|h\|^q,
\end{equation}
for all $u,h\in X$. Further, we say that it is \emph{monotone} if \eqref{eq:monot} holds for $c=0$; see for example \cite{Showalter}.

Two differentiability concepts are used in this work: Fr\'{e}chet,  and Newton (or slant) differentiability. For the definition of the Fr\'{e}chet one we refer the reader to \cite{curtain1977functional} or a nonlinear functional analysis book. We introduce now the Newton differentiability concept; see \cite{Ito2008} for a solid introduction on the subject. Let $X, Y$ be real Banach spaces and $D\subset X$ be an open set.
Then $F: D\subset X\to Y$ is called \emph{Newton differentiable} at $u$ if there exists an open neighborhood $\mathcal{N}(u)\subset D$ and mappings 
\begin{equation*}
	G_F:\mathcal{N}(u)\to \mathcal{L}(X, Y)
\end{equation*} 
such that
\begin{equation*}
	\lim_{\|h\|_X\to 0} \frac{\|F(u+h) - F(u) - G_F(u+h)h\|_Y}{\|h\|_X} = 0.
\end{equation*}
Note that Newton derivatives are in general not unique, and it is direct to prove that if  
$F: D\subset X\to Y$ is continuously Fr\'{e}chet differentiable, then it is Newton differentiable. 

In order to simplify notation, we use the Landau $o$ notation: We denote $\|r(h)\|_Y=o(\|h\|_X)$ for a map $r:X\to Y$ if the following holds true
\begin{equation*}
	\lim_{\|h\|_X\to 0} \frac{\|r(h)\|_Y}{\|h\|_X}=0,
\end{equation*}
i.e., $\|r(h)\|_Y=o(\|h\|_X)$ implies that $r$ vanishes  faster than $h$ as $h\to 0$.

We start now with a result that establishes that for a strongly monotone differentiable map, its derivative is also strongly monotone under relatively mild conditions.

\begin{lemma}
\label{thm:inheritedCoercivity}
Let \(A:D\subset X\to X'\) with $D$ open satisfy for some fixed $c>0$ and $2\leq q< 3$
\begin{equation}\label{eq:first}
\left\langle A (u+h) - A(u), h\right\rangle \geq c \|h\|^q,
\end{equation}
for all $u,u+h\in D$.

In addition, suppose that 
\begin{equation}\label{eq:second}
	\langle A (u+h) - A(u), h \rangle = \langle A'(u)h, h\rangle + \langle w(u,h), h \rangle
\end{equation}
with \(\|w(u,h)\|_{X^*}\leq M \|h\|_X^{2}\), for all $u\in D$, and all $h\in B_r(0,X)$ for some sufficiently small $r=r(u)>0$.

Then for each $u\in D$, \(A'(u)\) is strongly monotone with $q=2$, i.e.,
\begin{equation*}
	\langle A'(u)h, h \rangle \geq \tilde c \|h\|_X^2,
\end{equation*}
for some $\tilde c>0$ and all $h\in X$.
\end{lemma}

\begin{proof}
Let $u\in D$ be fixed. From \eqref{eq:first} and \eqref{eq:second} we obtain, for all $h\in B_s(0,X)$ with $s>0$ sufficiently small, that
\[
c \|h\|_X^q
\leq   \left\langle A'(u)h, h\right\rangle +M \|h\|_X^{3}.
\]
Let $h=t\tilde{h}$ for some $\tilde{h}\in B_s(0,X)$ with $\|\tilde{h}\|_X=s$, and $t\in [0,1]$, then
\begin{equation*}
    t^{q-2}c s^q
    \leq  \langle A'(u)\tilde h, \tilde h \rangle
    + Mts^{3}.
\end{equation*}
If $q=2$, simply take $t=0$ and the result follows. If $q\in (2,3)$, then define $g(t)=t^{q-2}c s^q-Mts^{3}$, and we choose $s$ sufficiently small so that the positive maximum of $g$ is achieved in $(0,1)$; this can be done since $q<3$. In fact, since $q\in (2,3)$, the maximum value is given by $g(t^*)=\tilde{c}s^2$ for some $\tilde{c}>0$ independent of $s$. Thus, $\tilde{c}s^2
    \leq  \langle A'(u)\tilde h, \tilde h \rangle$ and hence
\begin{equation*}
  \tilde{c}\|\tilde{h}\|^2_X  \leq  \langle A'(u)\tilde h, \tilde h \rangle.
\end{equation*}
Scaling by $\gamma>0$ and using that $A'(u)$ is linear, completes the proof.
\end{proof}

A few words are in order on the assumption in \eqref{eq:second}. This condition is satisfied if $A$ is twice Fr\'{e}chet differentiable on the open set $D$ and its second order derivative \(A''\) is uniformly bounded in \(D\); the result follows from the Taylor remainder theorem, see \cite{zeidler2013nonlinear}. 

We show next that in some cases the Newton derivative also inherits monotonicity properties of the original map. In particular, this requires some continuity assumption on the Newton derivative with respect to the base point. 

\begin{lemma}
\label{thm:Fmono}
Let  \(F:D\subset X\to X'\) be Newton differentiable with Newton derivative $G_F$ with $D$ open and suppose that $F$ is monotone, i.e.,
\begin{equation}\label{eq:Fmono}
\left\langle F (u+h) - F(u), h\right\rangle \geq  0,
\end{equation}
for all $u,u+h\in D$. 
If $w\mapsto \langle G_F(w)h,h\rangle$ is continuous at $w=u\in D$ for each $h$, then \(G_F(u)\) is monotone, i.e.,
%TODO: double "then"
\begin{equation*}
	\langle G_F(u)h, h \rangle \geq  0
\end{equation*}
for  all $h\in X$.
\end{lemma}

\begin{proof}
	By \eqref{eq:Fmono} and the definition of Newton derivative, we observe that
	\begin{equation*}
		\left\langle G_F(u+h)h, h\right\rangle \geq  \left\langle r(h), h\right\rangle,
	\end{equation*}
	for some $r(h)=o(\|h\|_X)$. Let $h=t\tilde{h}$ with $\tilde{h}\in X$ fixed, then 
		\begin{equation*}
		 \langle G_F(u+t\tilde{h})\tilde{h}, \tilde{h} \rangle \geq   \frac{\langle r(t\tilde{h}), t\tilde{h} \rangle}{\|t\tilde{h}\|^2_X}\|\tilde{h}\|^2_X ,
	\end{equation*}
	and the result follows by taking the limit of $t\to 0$.
\end{proof}

We now establish the main theorem of the section and the tool which later allows us to determine differentiability properties of the control-to-state map in the introduction.

\begin{theorem}
\label{thm:thm1}

Let \(A: X\to X'\)  satisfy the assumptions of Lemma \ref{thm:inheritedCoercivity} for $q=2$ and $D=X$, and \(F:X\to X'\) be monotone, continuous and Newton differentiable, and suppose that either $F$ satisfies the assumptions of Lemma \ref{thm:Fmono} or that its Newton derivative $G_F(u):X\to X'$ is monotone. 

Then, for \(f\in Y'\), $u(f)$ is well-defined as the unique solution  $u\in X$ to the equation 
\begin{equation}\label{eq:u_f_definition}
A(u) + F(u) = f \quad \text{ in }\: X'.
\end{equation}
In addition,
\begin{equation*}
	Y'\ni f\mapsto u(f)\in X 
\end{equation*}
is Newton differentiable  with Newton derivative $Y'\ni f\mapsto G_u(f)\in \mathcal{L}(Y',X)$ defined as: For $h\in Y'$, $w(f)=G_u(f)h$ is the unique solution $w\in X$ to
\begin{equation}\label{eq:w_candidate}
  A'(u(f)) w + G_F(u(f)) w = h \quad \text{ in } X'
\end{equation}
where \(A'\) is the Fréchet derivative of \(A\).
\end{theorem}

\begin{proof}
Note first that $u(f)$ is well-defined for any $f\in Y'$ given that  there exists a unique solution to \eqref{eq:u_f_definition}; the result follows by standard methods in monotone operator theory, see \cite[Theorem 2.1]{Showalter}. Further, by Lemma \ref{thm:inheritedCoercivity}, we have that for any $u$, $A'(u)$ is strongly monotone, and $G_F(u)$ is monotone by initial assumption or Lemma \ref{thm:Fmono}, so that $w$, the unique solution to \eqref{eq:w_candidate}, is also well-defined; again by \cite[Theorem 2.1]{Showalter}. 

Since \(A\) is continuously Fréchet differentiable, it also is Newton differentiable. Then \(E:=A+F\) is Newton differentiable with derivative $G_E:=A'+G_F$ satisfying 
\begin{equation*}
	\langle G_E(u)v,v\rangle \geq \langle A'(u)v,v\rangle\geq \tilde{c} \|v\|_X^2,
\end{equation*} 
for all $u,v\in X$.

For any \(f,h\in Y'\), define \(d( h) = u(f+h) - u(f)\). Considering \eqref{eq:u_f_definition} with \(f\) and \(f+h\) and subtracting the results, we obtain  
\[
E(u(f) + d(h)) - E(u(f)) = h.
\]
Since $A$ satisfies \eqref{eq:first} with $q=2$, and $F$ is monotone, it follows that \(Y'\ni f\mapsto u(f)\in X\) is Lipschitz continuous. Then, for a map $r:Y'\to X'$ satisfying \(\|r(h)\|_{X'}=o(\|d( h)\|_{X})\), we have \(\|r(h)\|_{X'}=o(\|h\|_{Y'})\). Subsequently, from the definition of Newton derivative,
\begin{equation}\label{eq:Ederivative}
G_E(u(f)+d(h))d(h) = h + r(h),
\end{equation}
where 
\begin{equation*}
	\|r(h)\|_{X'}=o(\|h\|_{Y'}).
\end{equation*}

In contrast, from \eqref{eq:w_candidate}, for some $w=w(f+h)$ we have that
\[
G_E(u(f+h))w(f+h) = h,
\]
and substracting this from \cref{eq:Ederivative}, we obtain
\begin{equation}
	\label{eq:testR}
G_E(u(f+h))R(h) = r(h),
\end{equation}
where $$R(h) := u(f+h) - u(f) - w(f+h).$$ By testing in \eqref{eq:testR} with $R(h)$,  we observe due to the strong monotonicity of $G_E(u(f+h))$ that
\[
\tilde{c}\|R(h)\|_X^2\leq \langle G(u(f+h))R(h), R(h)\rangle = \langle r(h), R(h) \rangle,
\]
and since 
\begin{equation*}
	\langle r(h), R(h) \rangle\leq \|r(h)\|_{X'}\|R(h)\|_X ,
\end{equation*}
we observe
\begin{equation*}
	\|R(h)\|_X=o(\|h\|_{Y'}),
\end{equation*}
and the proof is finished. 
\end{proof}
 
We now aim at applying the results in this section to our sandpile control problem.

\section{Application to the sandpile control problem}\label{sec:appl}
In the framework of \cref{eq:control_problem}, we consider the problem in the introduction associated to the control of the stationary accumulation of granular material.
In this section, we fix $X=H_0^1(\Omega)$, and  $X\subset Y\subset L^2(\Omega)$ with $Y$ a real Banach space, e.g., $Y=L^2(\Omega)$.  Note that this implies that $L^2(\Omega)\subset Y'\subset X'$ since we identify $L^2(\Omega)$ with its topological dual.

In \eqref{eq:reg}, we define the constraint regularization operator $\pcal : X \to X'$ as

%	\begin{align}
%		\langle\pcal(u), w\rangle_{X',X}\label{def:calP}
%		&= \int_{\Omega^+(u)} P(\mathrm{D} u)\cdot \mathrm{D} w\,\textup{d}x\\\notag
%		&= \int_{\Omega^+(u)}(|\mathrm{D} u| - \varphi)^\pm \frac{(\mathrm{D} u \cdot \mathrm{D} w)}{|\mathrm{D} u|}\,\textup{d}x,
%	\end{align}
\begin{equation}\label{def:calP}
		\langle\pcal(u), w\rangle_{X',X}
		= \int_{\Omega^+(u)} P(\mathrm{D} u)\cdot \mathrm{D} w\,\textup{d}x
		= \int_{\Omega^+(u)}(|\mathrm{D} u| - \varphi)^\pm \frac{(\mathrm{D} u \cdot \mathrm{D} w)}{|\mathrm{D} u|}\,\textup{d}x,
\end{equation}
where 
\begin{equation*}
	\Omega^+(u) := \{x\in\Omega : |\rmD u(x)|>0 \:\:\text{ a.e.}\}.
\end{equation*}
Additionally, $P(u) := q(u)b(u)$ for
\begin{equation*}
	q(u) := \frac{u}{|u|}, \quad\text{ and }\quad  b(u) := (|u| - \varphi)^\pm,
\end{equation*}
and $(\cdot)^\pm := \min(1, \max(0, \cdot))$ in the pointwise sense: for $g:\Omega \to \mathbb{R}$, then 
\begin{equation*}
	(g(x))^\pm=\begin{cases}
		0,& \qquad \text{ if } g(x)\leq 0,\\
		g(x),& \qquad \text{ if } 0\leq g(x)\leq 1,\\
		1,& \qquad \text{ if } 1\leq g(x).
	\end{cases}
\end{equation*}
 
Two possible choices for $\mathrm{D} $ are considered: $\mathrm{D} = \nabla$, the weak gradient,
and $\mathrm{D} = \gradapprox$,
where $\gradapprox : L^p(\Omega)\to L^p(\Omega)^\mathrm{d}$ is a bounded linear operator for all $1\leq p\leq +\infty$, approximating the gradient (e.g., by means of incremental quotients). The parameter $\mu>0$ can be considered to obtain $\gradapprox\to \nabla$ in some sense, as $\mu\downarrow 0$. In addition, note that (formally) we can write $\mathcal{P}(u)=\mathrm{D}'P(\mathrm{D}u)$.
 
 For both cases, we have that $\mathcal{P}$ corresponds to the derivative of the convex functional
\begin{equation*}
	J_{\mathcal{P}}(u)= \int_\Omega   K(|\mathrm{D} u(x)| - \varphi(x)) \textup{d}x,
\end{equation*}
where 
\begin{equation*}
	K(t)=\begin{cases}
		\int_0^t(y)^\pm\mathrm{d}y,& \quad \text{ if } t\geq 0;\\
		0,& \quad \text{ if } t< 0.
	\end{cases}
\end{equation*}
It follows that $\mathcal{P}$ is monotone given that $J_{\mathcal{P}}$ is convex, and further from its definition we observe that  $\mathcal{P}$ is continuous. Thus,  in both the cases $\mathrm{D} \in\{ \nabla, \gradapprox\}$, we obtain that for every $f\in Y'$, since $-\Delta$ is strongly monotone, equation \eqref{eq:reg} has a unique solution by the same argument as in the beginning of the proof of Theorem \ref{thm:thm1}. Hence, for each $\gamma>0$, there exists $u_\gamma\in X$ solution to \eqref{eq:reg}, and standard arguments exploiting the monotonicity of $J_{\mathcal{P}}$ determine that $u_\gamma\to u^*$ in $X$ as $\gamma\to \infty$, where $u^*$ is the solution to \eqref{eq:original}; see for example \cite{hr12}.

Next we show that $P(u)=q(u)b(u)$  is Newton differentiable between appropriate spaces, and later we use the result to obtain a Newton differentiability result for $\mathcal{P}(u)=\mathrm{D}'P(\mathrm{D}u)$.

\begin{lemma}\label{lm:Gp_Newton_derivative}
	The operator $P: L^p(\Omega)^\mathrm{d}\to L^q(\Omega)^{\mathrm{d}}$ with $2\leq 2q \leq p $
	is Newton differentiable. A Newton derivative $G_P$, can be defined as
	\begin{equation}\label{eq:G_P}
		G_P(u) = q(u)G_b(u) + b(u)Q(u),
	\end{equation}
	where
	\[
		Q(u)(x) := \frac{1}{|u(x)|}\left(\mathrm{id} - \frac{u(x)u^T(x)}{|u(x)|^2}\right),
	\]
	and
	\[
		G_b(u)(x) := G_{\max}^{\min}(|u(x)|-\varphi(x))\frac{u^T(x)}{|u(x)|}
	\]
	is a Newton derivative of $b: L^p(\Omega)^{\mathrm{d}}\to L^q(\Omega)$, and
	\begin{equation*}
		G^{\min}_{\max}(w)(x) := \chi_{(0,1)}(w(x))
	\end{equation*}
  is a Newton derivative of $(\cdot)^\pm: L^p(\Omega)\to L^{q}(\Omega)$.
\end{lemma}

\begin{proof}
The fact that $G_b(u)$ is a Newton derivative of $b: L^p(\Omega)^{\mathrm{d}}\to L^q(\Omega)$ follows from
  $G^{\min}_{\max}(u)(x) = \chi_{(0,1)}(u(x))$  being  a Newton derivative of $(\cdot)^\pm: L^p(\Omega)\to L^{q}(\Omega)$, and this is
similarly obtained as the Newton derivative of the $\max$ function alone, see \cite{Ito2008} and compare to \cite{hr12}.

Initially, we observe that
\begin{align*}
	P(u+h) - P(u) - G_P(u+h)h
=\, &b(u+h)\left(q(u+h) - q(u) - Q(u+h)h\right)\\
&+\left(q(u)-q(u+h)\right)\left(b(u+h)-b(u)\right)\\
&+q(u+h)\left(b(u+h)-b(u)-G_b(u+h)h\right)\\
=\, &I + II + III,
\end{align*}
and in what follows we show that 
\begin{equation*}
	I+II+III = o(\|h\|_p).
\end{equation*}

{\it Consider initially $I$}.
Note that we have
%\begin{align*}
\begin{equation*}
	I = b(u+h) \left( - \frac{u}{|u+h||u|}\left(|u+h| - |u| - \frac{(u+h)^Th}{|u+h|}\right)
												+ \frac{(u+h)^Th}{|u+h|^2}\left(\frac{u+h}{|u+h|} - \frac{u}{|u|}\right)\right).
\end{equation*}
%\end{align*}
Since $$\left|b(u+h) \frac{u}{|u+h||u|}\right|\leq \frac{1}{\nu},$$ due to the fact that $\varphi\geq \nu$ a.e. in $\Omega$,
and because 
\begin{equation*}
	G_{|\cdot|}(u+h)h=\frac{(u+h)^Th}{|u+h|},
\end{equation*}
when $u+h\geq \nu$, where $G_{|\cdot|}$ is a Newton derivative of $|\cdot|: L^p(\Omega)\to L^q(\Omega)$, we observe that the first part of the sum in $I$ is $o(\|h\|_p)$. 

Regarding the second term, if $u+h\leq \nu$ then $b(u+h)=0$, and if $u+h\geq \nu$, then
\begin{equation*}
	\left|\frac{(u+h)^Th}{|u+h|^2}\right|\leq \frac{|h|}{\nu} \: \text{ and } \:\left(\frac{u+h}{|u+h|} - \frac{u}{|u|}\right) \leq 2\frac{|h|}{\nu},
\end{equation*}
and since $b(u+h)\leq 1$, we get a bound of $2\frac{|h|^2}{\nu^2}$. An application of Hölder's inequality implies that the second term of $I$ is also $ o(\|h\|_p)$ as we see next. Suppose that $z:\Omega\to\mathbb{R}$ satisfies $|z|\leq |h|^2.$
	Thus by applying Hölder's inqualitiy to $|h|^{2q}$ and $1$ for the exponents $\frac{p}{2q}\geq 1$ and $\left(\frac{p}{2q}\right)'$ we get
	\[
	\int_\Omega |z|^q \mathrm{d}x\leq \int_\Omega |h|^{2q}\mathrm{d}x
	%\leq \left(\int_\Omega 1\right)^{{1/\left(\frac{p}{2q}\right)'}} \left(\int_\Omega|h(x)|^p\right)^{2q/p}
	\leq C_1(\Omega) \left(\int_\Omega|h|^p\mathrm{d}x\right)^{2q/p}
	\]
	Therefore (by taking the $q$-th root), we get
	\[
		\|z\|_{L^q(\Omega)} \leq C_2(\Omega)  \left(\int_\Omega|h|^p\mathrm{d}x\right)^{2/p} = C_2(\Omega)  \|h\|^2_{L^p(\Omega)},
	\]
which proves the statement.

{\it We turn our attention now to $II$}.
For a.e. \(x\in\Omega\) it holds that
\begin{equation}
	\left| b(u(x)+h(x)) - b(u(x))\right|\leq |h(x)|  \chi_{\Omega_\nu}(x)
\end{equation}
where
\[
	\Omega_\nu := \{x\in\Omega : |u(x)|<\nu \land |u(x)+h(x)|<\nu \:\:\: \text{a.e.}\}.
\]
Thus for \(x\in\Omega\setminus\Omega_\nu\) we observe that
\[
  \left| q(u) - q(v) \right| \leq 2\min\left(\frac{|h|}{|u|}, \frac{|h|}{|u+h|}\right) \leq 2\frac{|h|}{\nu},
  \]
	thus $II$ is bounded by $ |h|^2/\nu$.
	As above, this implies $II = o(\|h\|_p)$.

{\it Finally, we consider $III$}.
Since $G_b$ is the Newton derivative of $b : L^p(\Omega)^\mathrm{d}\to L^q(\Omega)$, and $|q|$ is bounded by $1$, we have that $III = o(\|h\|_p)$.

\end{proof}

Let $\pcal^{\nabla}:X\to X'$ and $\pcal^{\gradapprox}:X\to X'$ be defined as $\mathcal{P}$ in \eqref{def:calP} for $\mathrm{D}=\nabla$ and $\mathrm{D}=\gradapprox$, respectively. Although these operators are well-defined as maps from $X$ to $X'$, in order to obtain a differentiability properties in the case of $\pcal^{\nabla}$, the operator needs to be defined in slightly different spaces as we see next.

\begin{lemma}\label{cor:pcal_newtondifferentiable}
The maps $\pcal^{\nabla} $ and $\pcal^{\gradapprox} $ are Newton differentiable when defined as $\pcal^{\nabla}: X \to (W_0^{1,\infty}(\Omega))'$
	and $\pcal^{\gradapprox}: X\to X'$. The general expression of a Newton derivative  $G_\pcal$ in these cases 
	 is given by
  \begin{equation}
  	\label{eq:G_Pcal}
  	\langle G_\pcal(u)v, w \rangle = \int_{\Omega^+(u)}(G_P(\mathrm{D} u)\mathrm{D} v)\cdot \mathrm{D} w\,\textup{d}x,
  \end{equation}
	for
	\begin{itemize}
		\item[(i)] all $u, v \in X$, $w\in W_0^{1,\infty}(\Omega)$, and  the duality pairing considered between $(W_0^{1,\infty}(\Omega))'$ and $W_0^{1,\infty}(\Omega)$ in the case $\mathrm{D}=\nabla $ and $\pcal=\pcal^\nabla$.
				\item[(ii)] all $u, v,w \in X$, and  the duality pairing considered between $X'$ and $X$ in the case of $\mathrm{D}=\gradapprox$ and $\pcal=\pcal^\gradapprox$.
	\end{itemize}
\end{lemma}

\begin{proof}
Consider (i) first. The map $\nabla:X\to L^2(\Omega)^\mathrm{d}$ is Fr\'{e}chet differentiable with derivative $\nabla$, and by Lemma \ref{lm:Gp_Newton_derivative} the map $P:L^2(\Omega)^\mathrm{d}\to L^1(\Omega)^\mathrm{d}$ is Newton differentiable. Then  $u\mapsto P(\nabla u)$ is Newton differentiable (since it is the composition of a Newton and a Fr\'{e}chet differentiable mapping  \cite{Ito2008}) as map from $X\to L^1(\Omega)^\mathrm{d}$ with Newton derivative $u\mapsto G_P(\nabla u)\nabla$. From here, and application of H\"{o}lder's inequality can be used to show that $\mathcal{P}^\nabla:X \to (W_0^{1,\infty}(\Omega))'$ is Newton differentiable (analogously as in \cite[Corollary A.3]{hr12}), with Newton derivative given by \eqref{eq:G_Pcal}.

Next consider (ii). In the case of $\gradapprox$, we use that for $\mathrm{d}=1,2$, by the Sobolev Embedding Theorem (e.g. see \cite{adams_fournier_2003}), $X$ is embedded in $L^p(\Omega)$ for any $2\leq p<\infty$, and by the same result we have that $L^q(\Omega)$ is continuously embedded in $H^{-1}(\Omega)$ for $q>1$. In addition,  $\gradapprox :X\to L^p(\Omega)^\mathrm{d}$ is Fr\'{e}chet differentiable with derivative $\gradapprox$ for $2\leq p<\infty$, and  $P: L^p(\Omega)^d\to L^q(\Omega)$ with $2\leq 2q \leq p $ is Newton differentiable. Then, as in the previous item, $u\mapsto P(\gradapprox u)$ is Newton differentiable as map from $X\to L^q(\Omega)$ with Newton derivative $u\mapsto G_P(\gradapprox u)\gradapprox$. By choosing a $q>1$, an application of H\"{o}lder's inequality shows that $\mathcal{P}^\nabla:X \to X'$ is Newton differentiable with Newton derivative given by \eqref{eq:G_Pcal}.
\end{proof}

Next we prove the existence of a solution to the sensitivity equation: Note that this requires to show the monotonicity of $G_\pcal(u)$  directly. The reason for this is that   $\mathcal{P}$ does not satisfy the continuity assumption of Lemma \ref{thm:Fmono}, i.e., continuity of $w\mapsto \langle G_{\mathcal{P}}(w)h,h\rangle$.

\begin{lemma}\label{prop:sensitivity_existence}
	There exists a unique $w\in X$ such that 
\begin{equation}\label{eq:derivative_eqn}
	\langle-\epsilon\Delta w, z\rangle_{X', X}
	+ \langle G_{\pcal}(u)w, z\rangle_{X', X}
	= \langle h, z\rangle_{X', X}
\end{equation}
for all $z\in X$, for $\pcal=\pcal^\nabla$ and for $\pcal=\pcal^\gradapprox$.
\end{lemma}

\begin{proof}
	%The operator $-\Delta$ is continuously Fréchet differentiable, hence $D\Delta(u) = 
	The map $-\Delta $ is strongly monotone and linear, then we only need to show  that $G_{\pcal}(u)\in \mathcal{L}(X, X')$ is monotone, i.e.,
	\[
		\langle G_\pcal(u) z, z\rangle_{X', X}\geq 0,
	\]
	for all $z\in X$.

	Note that $|G_P(\rmD u)|\in L^\infty(\Omega)$ for each $u\in X$, then we are only left to prove that $\rmD 'G_P(\rmD u) \rmD \geq 0$. Exploiting the structure of $G_P(\rmD u)$ of \cref{eq:G_P} we have
	%\begin{align*}
	\begin{equation*}
		\langle G_\pcal(u) z, z\rangle_{X', X} =
		(q(\rmD u)G_b(\rmD u)\rmD z, \rmD z) + ( b(\rmD u)Q(\rmD u)\rmD z, \rmD z).
	\end{equation*}
	%\end{align*}
For the first term we have
	\[
		G_{\min}^{\max}(|\rmD u| - \varphi)
		\frac{(\rmD u)^T \rmD z}{|\rmD u|^2}
		\cdot \frac{(\rmD u)^T \rmD z}{|\rmD u|^2}\geq 0,
	\]
	and since $b(\rmD u)\geq 0$ and
	\[
		Q(\rmD u)\rmD z \cdotp \rmD z = \frac{1}{|\rmD u|}\left(|D z|^2 - \frac{\left((\rmD u)^T\rmD z\right)^2}{|\rmD u |^2}\right) \geq 0,
	\]
  the second term is also monotone and the proof is complete.
\end{proof}

% TODO: say, how we apply the abstract results to this case.
We are now in shape to apply the results from \cref{sec:abstract_results} to this specific control problem.
\begin{theorem}\label{thm:ftou}
	For $\pcal = \pcal^\gradapprox$,
	the solution map $Y'\ni f\mapsto u(f)\in X$ of \cref{eq:reg} is Newton differentiable.
\end{theorem}

\begin{proof}
	The result follows as direct application of the abstract result in \cref{thm:thm1} where the hypotheses are covered by Lemma \ref{cor:pcal_newtondifferentiable} and \ref{prop:sensitivity_existence}.
\end{proof}

A significant obstacle of the ``non-$\epsilon$ differentiability gap'' in Lemma \ref{lm:Gp_Newton_derivative} results in the lack of an analogous result for the case $\pcal = \pcal^\nabla$. However, some of these issues can be resolved in the algorithmic area with the introduction of a ``lifting operator'', see \cite{MR3333651}.

A direct corollary of the previous result concerns the Newton differentiability of the reduced functional for the control problem under study.

\begin{corollary}
	Provided that $Y'\ni f\mapsto \|f\|^2_{Y'}$ is Newton differentiable, the  functional 	$Y'\ni f\mapsto  J(u(f), f)$ where $J$ is defined in \eqref{eq:J} is Newton differentiable.
\end{corollary}

\begin{proof}
	Note that $u\mapsto \frac12\int_\Omega|u-u^d|^2\,\textup{d}x$ is Fr\'{e}chet differentiable. Then, the proof is an application of the previous theorem and the composition result of Newton differentiable maps in \cite{hintermuller2010semismooth}.
\end{proof}

\section{Solution algorithms}
The idea of this section is to provide a solid idea of the applicability of the results of the previous section for the development of solution  algorithms for \eqref{eq:reg} and \eqref{eq:control_problem}.

As defined previously, the map $E:X\to X'$ is given by
\[
	E (u) = -\epsilon\Delta u +\gamma\pcal(u) - f,
	%E(u) = -\epsilon\Delta u +\gamma\pcal(u) - f.
\]
so that the state equation \cref{eq:reg}  
can be written as 
\begin{equation*}
	E(u) = 0.
\end{equation*} 
It is known that  for every $f\in H^{-1}(\Omega)$, this equation is uniquely solvable in all cases contemplated in this paper, namely for $\mathrm{D}=\nabla$ and $\mathrm{D}=\gradapprox$.

In the case $\mathrm{D}=\gradapprox$, and given that $E$ is Newton differentiable as a map $X\to X'$, a Newton derivative of $E$ is given by
\[
	G_{E}(u)v = -\epsilon\Delta v + \gamma G_\pcal(u) v,
\]
where $G_\pcal$ is made explicit in \eqref{eq:G_Pcal}. Further, since the constant in the strong monotonicity of $-\epsilon\Delta $ is identical to $\epsilon$, and we have proven that $G_\pcal(u)$ is monotone. It follows that  $G_E(u)$ is nonsingular and its inverse
is uniformly bounded by $\epsilon^{-1}$, i.e.,
\[
	\|G(u)^{-1}\|\leq\frac{1}{\epsilon}.
\]
Hence, solutions to $E(u)=0$ are suitable to be approximated by a function space version of a \emph{semismooth Newton method}; see  \cite[Theorem 8.16]{Ito2008}.

The semismooth Newton iteration is then given by $$u^{n+1} := u^n + v^\ast,$$ where the Newton step $v^\ast$ is defined as the solution of
\begin{equation}\label{eq:newton_step}
	G_E(u^n) v = - E(u^n).
\end{equation}
This yields (see \cite[Theorem 8.16]{Ito2008}) that
the sequence $\{u^{n}\}_{n\in\mathbb{N}}$ defined by the iteration \cref{eq:newton_step}  converges superlinearly to the unique solution $u^\ast\in X$ of the state equation $E(u) = 0$,
provided that the initial iterate $u_0$ is sufficiently close to $u^\ast$. The fact that $\mathrm{D}=\gradapprox$ is considered, leads to the idea to further apply a path-following method simultaneously on $\gamma$ and $\mu$.

Under mild conditions we have proven that $Y'\ni f\mapsto  J(u(f), f)$ is Newton differentiable. In the case where $Y'$ does not have a large number of degrees of freedom, a descent approach for the overall problem is possible. However the computation of the entire derivative is prohibitive if  the dimension of $Y'$ is not small. On the other hand, if the approach is such that the entire derivative is not needed (only a small number of components are),  a descent method is directly suitable.

\section{Conclusion and future research}

We have provided several abstract results that link notions of differentiability to monotonicity ones. In fact, in Theorem \ref{thm:thm1}, the hypotheses contemplate the possibility of heavily nonlinear operators, e.g., the $p-$Laplacian defined as $-\Delta_pu=-\mathrm{div}(|\nabla u|^{p-2}\nabla u)$ with $p\in [2,3)$ is within the scope of the theorem. The application of the  abstract results to the specific example of control of the pile of granular material is then tackled in Theorem \ref{thm:ftou} and its corollary. There, a function space approach is suitable provided that the gradient is considered in its approximated version. The result of Newton differentiability on Lemma \ref{cor:pcal_newtondifferentiable} is of interest in its own right; it provides a better result than the one standing in the literature for the case $\mathrm{D}=\nabla$. The short discussion on the algorithmic development is a current area of active research.

\bibliographystyle{abbrv}
\bibliography{references}

\end{document}